\documentclass[preprint,12pt]{elsarticle}

\usepackage{amssymb}
\usepackage{amsthm}

\usepackage{hyperref}

\theoremstyle{definition}
\newtheorem{definition}{Definition}
\theoremstyle{plain}
\newtheorem{thm}{Theorem}
\newtheorem{lemma}{Lemma}
\newtheorem{cor}{Corollary}

\newcommand{\R}{\mathbb R}

\newcommand{\order}{h} 
\newcommand{\ordertwo}{\tilde h} 
\newcommand{\HLOM}{G} 
\newcommand{\desiredoutput}{y^*} 
\newcommand{\desiredoutputmatrix}{Y^*} 

\journal{Neurocomputing}

\begin{document}

\begin{frontmatter}



\title{Universal Approximation Using Shuffled Linear Models}


\author{Laurens Bliek}

\address{Faculty of Electrical Engineering, Mathematics, and
Computer Science,
Delft University of Technology,
Mekelweg 4, 2628 CD  Delft, The Netherlands}

\begin{abstract}
This paper proposes a specific type of Local Linear Model, the Shuffled Linear Model (SLM), that can be used as a universal approximator. Local operating points are chosen randomly and linear models are used to approximate a function or system around these points. The model can also be interpreted as an extension to Extreme Learning Machines with Radial Basis Function nodes, or as a specific way of using Takagi-Sugeno fuzzy models. Using the available theory of Extreme Learning Machines, universal approximation of the SLM and an upper bound on the number of models are proved mathematically, and an efficient algorithm is proposed.
\end{abstract}


\end{frontmatter}



\section{Introduction}

The approximation of nonlinear functions or systems is a problem that appears in many disciplines~\cite{nelles2001nonlinear} and that can be tackled by various approaches. One of these approaches is to decompose the problem into smaller problems and use a combination of simple solutions that only work for a small part of the problem. The combination of these local models should then approximate the global function or system. The Local Linear Model (LLM) is a key example of this strategy, where a dynamical system is partitioned into multiple operating regimes, and for each regime a simple linear model is used to model the system.

If the goal is to approximate a function, not to model a system, the use of artificial neural networks has been a popular approach. One of the most popular algorithms for approximating a function with neural networks is the backpropagation algorithm. The Extreme Learning Machine (ELM) algorithm as proposed by Huang et al.~\cite{huang2006extreme}, is an alternative algorithm which outperforms the backpropagation algorithm in many aspects. Instead of tuning the parameters of the hidden layer and output layer, the hidden layer parameters are initialised randomly and kept fixed while the output weights are trained by a linear least-squares method. Since very efficient algorithms are available for solving linear least-squares problems, the ELM algorithm is an efficient method for training a feedforward neural network.

Although ELMs can use nodes with many functions that do not at all resemble the activation functions observed in biological neurons~\cite{huang2006universal}, which makes ELM a very general approach, to the author's knowledge they have not yet been extended in such a way that they include LLMs. The activation functions can be very general, but the hidden network nodes typically take one of two forms: either additive nodes, or Radial Basis Function (RBF) nodes are used, while the output is a linear combination of these. 
Unfortunately, with this architecture, this type of network can not be made equivalent to a LLM as defined in this paper. To prove that the same ELM algorithm does work for LLMs, an extension has to be made.

The ELM algorithm is extended in this paper to include the use of LLMs, and the approximation capabilities are analysed and proved mathematically. Since ELMs use randomly initialised hidden layer parameters, the linear models used in the LLM in this paper will be accurate around a randomly chosen local point. This can be seen as shuffling a deck of cards and handing them to the players in a cardgame, where the cards consist of parameters that indicate the location and impact factors of the local points, while the players are the linear models. The proposed model in this paper is therefore called Shuffled Linear Model (SLM). Note that not the linear models themselves, but the localities where they are most valid are chosen randomly.

Extending the ELM algorithm to include LLMs gives several advantages, like a better physical interpretation and the potential to identify nonlinear systems, making the algorithm more fit for system identification. Besides this, the extension could be used in the area of fuzzy modelling. Fuzzy models are a popular approach to model nonlinear systems, and can be used to combine both rigid mathematical concepts and vaguer linguistic concepts. One of the most popular fuzzy models is the Takagi-Sugeno model~\cite{takagi1985fuzzy} (TSM). TSMs use linear consequences in their fuzzy rules, making them equivalent to LLMs under some circumstances. This is also shortly analysed in this paper, but the main point of this paper is the extension of the ELM algorithm.

The paper has the following structure: in Section \ref{sec:ELM}, a summary of the ELM approach will be given. Section \ref{sec:SLM} contains the proposed model with the definition of the SLM. Section \ref{sec:UA} provides proof of the approximation capabilities of the SLM and contains the proposed algorithm. Section \ref{sec:int} contains several ways to interpret the SLM. Conclusions are given in Section \ref{sec:concl}.

\section{ELM summary}
\label{sec:ELM}

The ELM architecture is similar to a feedforward neural network with one hidden layer. The output neurons have a linear activation function, while the hidden neurons have a nonlinear activation function. The main difference with traditional neural networks is in the tuning of the parameters: the hidden layer parameters are initialised randomly and remain fixed, while only the output weights are tuned. This leads to a linear least-squares problem.

Although the original ELM used neural alike hidden nodes, ELMs have been generalised to include many activation functions that are not neuron alike~\cite{huang2007convex}. The most common example is the \emph{Radial Basis Function} (RBF) $g(a,b,x) = e^{-b||x-a||^2}$, with parameters $a\in \mathbb R^n$, $b\in \mathbb R_{> 0}$, and input $x\in \mathbb R^n$. In this paper, only ELMs with RBF nodes are considered. For a network with $\order$ hidden nodes, the output $y \in \mathbb R^m$ of the ELM is:
\begin{equation}
	y = \sum_{i=1}^\order{\beta_i g(a_i,b_i,x)} = \sum_{i=1}^\order{\beta_i g_i(x)},
\end{equation}
with $\beta_i \in \mathbb R^m$ the output weight vector connecting the $i$-th hidden node with the $m$ output nodes. For $N$ inputs $x_j$, this can be written in matrix notation as
\begin{equation}
	Y = \HLOM B,
\end{equation}
where
\begin{equation}\label{eq:HLOM}
	\HLOM = \left[\begin{array}{ccc}
			g_1(x_1) & \ldots & g_\order(x_1)\\
			\vdots & & \vdots \\
			g_1(x_N) & \ldots & g_\order(x_N)
			\end{array}
			\right]_{N\times \order}
\end{equation}
and
\begin{equation}
	Y = \left[\begin{array}{c}
			y_1^T\\
			\vdots\\
			y_N^T
			\end{array}
			\right]_{N\times m},\quad
	B = \left[\begin{array}{c}
			\beta_1^T\\
			\vdots\\
			\beta_h^T
			\end{array}
			\right]_{\order\times m}.
\end{equation}

If a similar $N \times m$-matrix $\desiredoutputmatrix$ consisting of target output values $\desiredoutput$ is constructed, the goal is to minimise the sum of square errors:
\begin{eqnarray}\label{eq:error}
	E  =  \sum_{j=1}^N{||y_j - \desiredoutput_j||^2} = ||Y - \desiredoutputmatrix||_F^2 = || \HLOM B - \desiredoutputmatrix||_F^2,
\end{eqnarray}
where $||\cdot||_F$ denotes the Frobenius norm. The output weights $B$ that minimise $E$, can be found by using the Moore-Penrose pseudo-inverse~\cite{rao1971generalized}, here denoted as $\dagger$. This also makes sure that the norm of the output weights are minimised. Then the optimal solution is:
\begin{equation}
	\hat B = \HLOM^\dagger \desiredoutputmatrix.
\end{equation}

\section{Proposed model}
\label{sec:SLM}

The model proposed in this paper is an extension of the ELM model. To define the model, two other definitions are given first:
\begin{definition}\label{def:LM}
	A \emph{Linear Model} (LM) is a function of the form $LM_i(x) = \alpha_i x + \beta_i$, with $\alpha_i$ an $m\times n$-matrix and $\beta_i \in \mathbb R^m$.
\end{definition}


\begin{definition}\label{def:LLM}
	A \emph{Local Linear Model} (LLM) is a weighted sum of LMs, where the model weights $p_i$ can depend on the input: $LLM(x) =$  $ \sum_{i=1}^{\order}{p_i(x) LM_i(x)}$ $ = \sum_{i=1}^{\order}{p_i(x)(\alpha_i x + \beta_i)}$.
\end{definition}

Now the proposed model will be defined:
\begin{definition}\label{def:SLM}
	A \emph{Shuffled Linear Model} (SLM) is a LLM with RBF model weights $p_i(x) = g_i(x) = e^{-b_i||x-a_i||^2}$, where the parameters $a_i \in \mathbb R^n$ and $b_i \in \mathbb R_{> 0}$ are drawn from a continuous probability distribution (independent of the input): $SLM(x) = \sum_{i=1}^{\order}{e^{-b_i||x-a_i||^2}(\alpha_i x + \beta_i)}$.
\end{definition}


Note that the SLM is equivalent to the ELM if all $\alpha_i$ are zero. It is therefore expected that, for nonzero $\alpha_i$, the approximation capabilities of the SLM are better than those of the ELM if the number of local models is equal to the number of hidden nodes $\order$
. This is what will be proved in this paper.


\section{Approximation capabilities of SLM and proposed algorithm}\label{sec:UA}
In this section it will be proved that a SLM, as an extension of an ELM, can act as a universal approximator, and an upper bound on the number of models will be given. First, it will be proved that the SLM can act as a universal approximator. After that, several lemmas will be proved that lead to the main theorem of this paper, where Lemma 2 is used in the proof of Lemma 3, and Lemmas 1 and 3 are used in the proof of the main theorem.  This section concludes with a learning algorithm for the SLM.

\subsection{Universal approximation with a SLM}

The result of the first theorem follows directly from the approximation capabilities of the ELM:

\begin{thm}
	SLMs are universal approximators. That is, given a continuous target function $f:\mathbb R^m \rightarrow \mathbb R^n$ and a SLM with free parameters $\alpha_i$ and $\beta_i$, $i = 1\ldots\order$, then with probability one the parameters $\alpha_i$ and $\beta_i$ can be chosen in such a way that $\lim_{\order\rightarrow \infty} \int_{\mathbb R^n}{||SLM(x) - f(x)||^2 dx} = 0$.
\end{thm}

\begin{proof}
	Choosing $\alpha_i = 0$, $i=1\ldots\order$, the SLM is equivalent to an ELM. For an ELM, the parameters $\beta_i$ can be determined using linear regression to guarantee universal approximation~\cite[Thm.~II.1]{huang2007convex}.
\end{proof}

The interesting part is, of course, whether the extension from ELM to SLM gives any improvements. It turns out that a SLM can also be trained using linear regression with all free parameters, just like an ELM. But for the same approximation accuracy, it needs less local models than the number of hidden nodes in the ELM. This is simply due to the fact that a SLM has more free parameters. The remainder of this section is devoted to the proof of these results, after which an algorithm for training the SLM will be presented.

\subsection{Necessary lemmas}
Before presenting the next theorem, some lemmas will be proved. Most proofs are similar to those of the ELM theorems~\cite{huang2006extreme,huang2007convex}.

\begin{lemma}\label{lem:linreg}
	Given a SLM with inputs $x_j$ and desired outputs $\desiredoutput_j$, $j=1\ldots N$, the error $E = \sum_{j=1}^N{||SLM(x_j) - \desiredoutput_j||^2}$ can be minimised using linear regression.
\end{lemma}

\begin{proof}
	Let $\gamma_i = [\alpha_i\ \beta_i]$ be the $m\times (n+1)$-matrix of the free parameters for model $i$, and let $\Gamma = [\gamma_1\  \ldots\  \gamma_{\order}]^T$ be the $(n+1)\order\times m$-matrix of all free parameters. Let also $z_j = \left[\begin{array}{c}x_j\\ 1\end{array}\right]$, then the output of the SLM can be written as:
	\begin{equation}
		SLM(x) = \sum_{i=1}^{\order}{g_i(x)\gamma_i z}.
	\end{equation}

	Now the error $E$ can be rewritten as follows:
	\begin{eqnarray}
		E & = &  \sum_{j=1}^N{||SLM(x_j) - \desiredoutput_j||^2}\nonumber\\
		& = &  \sum_{j=1}^N{||\sum_{i=1}^{\order}{g_i(x_j)\gamma_i z_j} - \desiredoutput_j||^2}\nonumber\\
		& = &  \sum_{j=1}^N{||\Gamma^T \left[\begin{array}{c}
																				g_1(x_j) z_j\\
																				\vdots\\
																				g_{\order}(x_j) z_j
																				\end{array}\right] 		- \desiredoutput_j||^2}\\
		& = & ||\Gamma^T K^T - {\desiredoutputmatrix}^T||_F^2,\nonumber\\
		& = & ||K \Gamma - \desiredoutputmatrix||_F^2,\nonumber
	\end{eqnarray}
	with
	\begin{equation}
		\desiredoutputmatrix = \left[\begin{array}{c} {\desiredoutput_1}^T\\ \vdots\\ {\desiredoutput_N}^T \end{array}\right]_{N\times m},
	\end{equation}
	and
	\begin{equation}\label{eq:matrixK}
		K = \left[\begin{array}{ccc}
							 g_1(x_1) z_1^T & \ldots & g_{\order}(x_1) z_1^T\\
							 \vdots & & \vdots\\
							 g_1(x_N) z_N^T & \ldots & g_{\order}(x_N) z_N^T\\
							 \end{array}\right]_{N\times (n+1)\order}.
	\end{equation}
	
	This sum of squared errors can be minimised using linear regression, for example by using the minimum norm least-square solution
	\begin{equation}
		\hat \Gamma = K^\dagger \desiredoutputmatrix.
	\end{equation}
	This completes the proof.
	
\end{proof}

\begin{lemma}\label{lem:distinctnorms}
	Given $N$ distinct samples $(x_j, \desiredoutput_j) \in \mathbb R^n \times \mathbb R^m$, if $a \in \mathbb R^n$ is chosen randomly from a continuous probability distribution, then with probability one we have $||x_j - a|| \neq ||x_{j'} - a||$, for $j\neq j'$, $j,j'=1 \ldots N$.
\end{lemma}

\begin{proof}
	The proof is an adaptation to one of the proofs of Huang and Babri~\cite{huang1998upper}. Consider the set $V(x_j, x_{j'}) = \{a\in \mathbb R^n: ||x_j - a|| = ||x_{j'} - a||\}$. This set is a hyperplane in $\mathbb R^n$. Since this is an $(n-1)$-dimensional surface in $\mathbb R^n$, and there are only finite $j,j'$, the union $U = \bigcup_{j\neq j'}{V(x_j, x_{j'})}$ is a finite union of $(n-1)$-dimensional surfaces for $j\neq j'$, $j,j'=1 \ldots N$. So, for any probability density function $f$, we have $P(a\in U) = \int_U{f(a) da} = 0$.
	Therefore, the probability that $||x_j - a|| = ||x_{j'} - a||$ is $0$ for randomly chosen $a$, and the result of the lemma follows.
\end{proof}

\begin{lemma}\label{lem:fullrank}
	For a SLM
	, matrix $K$ from equation (\ref{eq:matrixK}) has full rank with probability one if the matrix $Z = [z_1\ \ldots \ z_N]^T$ has full rank.
\end{lemma}

\begin{proof}
	For $i = 1\ldots \order$, and $k = 1\ldots n+1$, 
	the $(i\times k)$-th column of $K$ is given by $[g_i(x_1) z_{1_k} \ \ldots\ g_i(x_N) z_{N_k}]^T$, where $g_i(x) = g(a_i,b_i,x) = e^{-b_i||x-a_i||^2}$.
	Let $I=(u,v)$ be any interval from $\R$, with $u<v$. Let $c:I\rightarrow \R^N$ be the curve defined as
	\begin{equation}
		c(b_i) = [g_i(x_1) \ \ldots \ g_i(x_N)]^T,
	\end{equation}
	seen as a function of $b_i$, for $b_i \in I$.
	Using a similar proof as that of Tamura and Tateishi~\cite{tamura1997capabilities}, it can be proved by contradiction that $c(b_i)$ does not belong to a subspace with dimension less than $N$. 
	
	Suppose it does, then there exists a vector $w\neq 0$ orthogonal to this subspace:
	\begin{equation}
		w^T(c(b_i) - c(u)) = \sum_{j=1}^N{w_j g(a_i,b_i,x_j)} - w^T c(u) = 0.
	\end{equation}
	Without loss of generality, assume $w_N \neq 0$, then
	\begin{equation}
		g(a_i,b_i,x_N) = \frac 1{w_N} w^T c(u) - \sum_{j=1}^{N-1}{\frac{w_j}{w_N} g(a_i,b_i,x_j)}.
	\end{equation}
	 
	 On both sides of this equation are functions of $b_i$. Taking the $s$-th partial derivative to $b_i$ on both sides gives the following equation:
	 \begin{equation}\label{eq:inf}
	 	\frac{\partial^s}{\partial b_i^s} g(a_i,b_i,x_N) = -\sum_{j=1}^{N-1}{\frac{w_j}{w_N} \frac{\partial^s}{\partial b_i^s}g(a_i,b_i,x_j)}, \quad s=1, 2 ,\ldots.
	 \end{equation}
	 Now, $g$ is infinitely differentiable w.r.t. $b_i$, with $\frac{\partial^s}{\partial b_i^s} g(a_i,b_i,x_j) = (-||x_j-a_i||^2)^s e^{-b_i||x_j-a_i||^2}$. Using this fact and the result of Lemma
	 \ref{lem:distinctnorms}, it follows that with probability one, Equation (\ref{eq:inf}) actually contains an infinite number of equations that are linear in parameters $w_j$, for $s = 1,\ 2,\  \ldots$. The number of free parameters $w_j$, however, is $N$.
	 This gives a contradiction. Hence, $c(b_i)$ does not belong to a subspace with dimension less than $N$.
	 
	 Since the above contradiction holds for $b_i$ from any interval $I\subseteq \R$, it is possible to choose $\order$ values $b_1, b_2, \ldots, b_{\order}$ from any continuous probability distribution over $\R$, such that the following matrix $\HLOM$ has full rank:
	 \begin{equation}
	 	\HLOM = [c(b_1) \ c(b_2)\ \ldots\ c(b_{\order})].
	 \end{equation}
	 This is the same $N\times \order$ matrix as the hidden layer output matrix from equation \ref{eq:HLOM}. The matrix $Z = [z_1\ \ldots \ z_N]^T$ is an $N\times (n+1)$-matrix. Let $M = \HLOM \otimes Z$ be the Kronecker product of these matrices. Then $M$ is an $N^2\times (n+1)\order$-matrix. Since both $\HLOM$ and $Z$ have full rank, and for matrices the general rule $\mathrm{rank}(A\otimes B) = \mathrm{rank}(A)\mathrm{rank}(B)$ holds, the matrix $M$ has rank $r = \min(\order,N)\min(n+1,N)$.
	 
	 Now, matrix $K$ appears actually inside matrix $M$. For matrix $M$, we have:
	 \begin{equation}
	 	M = \HLOM \otimes Z = \left[\begin{array}{ccc}
	 																		g_1(x_1)z_1^T & \ldots & g_{\order}(x_1)z_1^T\\
	 																		\vdots && \vdots\\
	 																		g_1(x_1) z_N^T & \ldots & g_{\order}(x_1)z_N^T\\
	 																		\vdots && \vdots\\
	 																		g_1(x_N)z_1^T & \ldots & g_{\order}(x_N)z_1^T\\
	 																		\vdots && \vdots\\
	 																		g_1(x_N) z_N^T & \ldots & g_{\order}(x_N)z_N^T
	 																		\end{array}\right].
		\end{equation}
		The $N\times (n+1)\order$-matrix $K$ from equation (\ref{eq:matrixK}) follows from this matrix by removing rows until only the rows with indices $(j-1)(N+1)+1$ are left, for $j=1,\ldots, N$. It follows that the rank of matrix $K$ is equal to $\min(N,r) = \min(N,(n+1)\order)$, so matrix $K$ has full rank.
\end{proof}

\subsection{Improvement of the SLM as an extension of an ELM}

Using the Lemmas that appeared in this section, the main theorem of this paper, which is similar to that of Huang~\cite[Thm. 2.2]{huang2006extreme} can be presented. From the proofs of the Lemmas, the following definitions are needed:
\begin{eqnarray}
	z_j & = & \left[\begin{array}{c}x_j\\ 1\end{array}\right],\\
	Z  & = & [z_1\ \ldots \ z_N]^T.
\end{eqnarray}

\begin{thm}\label{thm:main}
	Given any $\epsilon >0$, and given $N$ distinct samples $(x_j,\desiredoutput_j) \in \R^n \times \R^m$, if matrix $Z$ has full rank, then there exists $\order \leq N/(n+1)$ such that a SLM with $\order$ local models can be trained using linear regression to get $E = \sum_{j=1}^N{||SLM(x_j) - \desiredoutput_j||^2}\  <\  \epsilon$ with probability one.
\end{thm}

\begin{proof}
	The error $E$ can actually become zero by choosing $\order = (n+1)N$. Since $Z$ has full rank, from Lemma \ref{lem:fullrank} it follows that matrix $K$ also has full rank, with probability one. Then $K$ is also invertible, since the size of $K$ is $N\times (n+1)\order = N\times N$. The solution showed at the end of the proof of Lemma \ref{lem:linreg} is now actually equal to $\hat \Gamma = K^{-1} \desiredoutputmatrix$, and the error $E = ||K\Gamma - \desiredoutputmatrix||_F^2$ is zero.
	
	Since zero error can be achieved by choosing $\order = (n+1)N$, there always exists a $\order \leq (n+1)N$ to let the error be as small as desired.
\end{proof}

Note that the condition that $Z$ has full rank can be easily satisfied, for example by adding noise to the training samples. Theorem \ref{thm:main} shows that the SLM needs less local models than the number of hidden nodes in an ELM. This does not necessarily imply a decrease in computation time, since the matrix of which a pseudo-inverse needs to be computed, has size $N\times (n+1)\order$, not size $N\times \order$. An ELM with $\order$ hidden nodes and a SLM with $\order/(n+1)$ local models are therefore comparable in both approximation capabilities and in computational efficiency. The construction of the matrix for which a pseudo-inverse needs to be computed, however, can be done more efficiently due to the block structure of $K$, so there is some gain in computational efficiency for a large number of models or hidden neurons. There might also exist a more efficient way to calculate the pseudo-inverse of this block matrix.

It is questionable whether it is desired to use less nodes or models in an approach that uses random nodes and models, 
since the model performance should not depend highly on parameters that are chosen randomly.

Still, there are several advantages of the SLM, compared to an ELM. As mentioned above, the matrix $K$ can be constructed efficiently because it contains a block structure, and maybe in the future an efficient method to calculate the pseudo-inverse of this matrix could be found. Besides this, the SLM is not as black-box as an ELM because the output matrix $\Gamma$ actually shows the direct (linear) relations between input and output near several operating points of the input space, whereas the output weights of an ELM only show relations between the less interpretable hidden layer and the outputs. Finally, the SLM allows for a clear physical interpretation due to its similarities with local linear models from the area of system identification, and its similarities with fuzzy inference models; see Sections \ref{sec:TSM} and \ref{sec:LLMint}.

\subsection{Proposed algorithm}

The proposed SLM learning algorithm is shown in Figure~\ref{fig:alg}. 
After running the algorithm, the output of the SLM for input $x$ is given by 
$SLM(x) = \Gamma^T \left[\begin{array}{c}
																				g_1(x) z\\
																				\vdots\\
																				g_{\ordertwo}(x) z
																				\end{array}\right],$ with $z =  \left[\begin{array}{c}x\\ 1\end{array}\right]$.\\

\begin{figure}[hbt]
\begin{center}
\fbox{\parbox{0.9\textwidth}{\begin{center}\textbf{SLM Algorithm}\end{center}\begin{enumerate}
	\item Given $N$ input-output pairs $(x_j,\desiredoutput_j)\in \R^n \times \R^m$, randomly generate parameters $a_i \in \mathbb R^n$ and $b_i \in \mathbb R_{> 0}$ from a continuous probability distribution independent of the input.
	\item Choose the number of models $\order$ and let $g_i(x) = e^{-b_i||x-a_i||^2}$ for $i=1,\ldots,\order$.
	\item For $z_j = \left[\begin{array}{c}x_j\\ 1\end{array}\right]$, calculate the $N\times (n+1)\order$-matrix\\ $K = \left[\begin{array}{ccc}
							 g_1(x_1) z_1^T & \ldots & g_{\order}(x_1) z_1^T\\
							 \vdots & & \vdots\\
							 g_1(x_N) z_N^T & \ldots & g_{\order}(x_N) z_N^T\\
							 \end{array}\right]$.
	\item Put the desired outputs $\desiredoutput_j$ in matrix $\desiredoutputmatrix = [\desiredoutput_1\ \ldots\ \desiredoutput_N]^T$ and calculate $\Gamma = K^\dagger \desiredoutputmatrix$ by using the Moore-Penrose pseudo-inverse.
\end{enumerate}}
}
\end{center}
\caption{Proposed SLM algorithm.}
\label{fig:alg}
\end{figure}


\section{Interpretation of the model}\label{sec:int}

Several interpretations of the SLM are possible. These interpretations will be described in the following subsections.

\subsection{ELM interpretation}\label{sec:elmint}

The main interpretation of the SLM in this paper is the extension of an ELM with RBF nodes. The output weights $\beta_i$ of an ELM can be seen as constant functions, since they do not depend directly on the input. After training an ELM, the output weights are the same for every possible input of the network. In the SLM extension, the output weights are not constants, but functions of the input. Instead of multiplying the outputs of the hidden layer $g_i(x)$ with constant output weights $\beta_i$, they are multiplied with a linear function of the input: $\alpha_i x + \beta_i$. This also changes the interpretation of output weight to linear model, and from hidden layer output to model weight.

Since only the parameters in the hidden layer are fixed, a SLM has more free parameters than an ELM. The vector $\beta_i$ is of size $m$ and the matrix $\alpha_i$ is of size $m\times n$. If $\order_1$ is the number of local models in a SLM and $\order_2$ the number of hidden neurons in an ELM, this implies that the SLM has $\order_1 \cdot m \cdot (n+1)$ free parameters, while an ELM has only $\order_2\cdot m$ free parameters. A logical conclusion is that an ELM needs $n+1$ times as many neurons as the number of local models in a SLM, for similar approximation capabilities. 
This is true under the conditions shown in Theorem \ref{thm:main}.

\subsection{RBF interpretation}

A SLM can be interpreted as a RBF network, but there are some important differences. In the first place, it is a generalisation of the RBF network. Several generalisations of RBF networks exist~\cite{billings2007generalized,fernandez2011melm}, but the one used in this paper is similar to the one proposed by Hunt, Haas and Murray-Smith~\cite{hunt1996extending}. Using the notation from this paper, a standard RBF network can be denoted as $y = \sum_{i=1}^\order{\beta_i g_i(x)}$, while a generalised RBF network is denoted as $y = \sum_{i=1}^{\order}{\theta_i(x) g_i(x)}$, letting the output weights depend on the input
\footnote{The generalised RBF as defined by Hunt, Haas and Murray-Smith, differs from a standard RBF network in three aspects. Only one of these aspects is considered here, namely that the output weights can depend on the input.}%
.\\ Using $\theta_i(x) = \alpha_i x + \beta_i$, where the output weights are a linear function of the input, this is equivalent to a SLM.

In the second place, a SLM differs from a RBF network in the way the hidden layer parameters are chosen. It is customary in RBF networks to let the location of the centers of the Gaussians $a_i$ be the same as some of the training samples $x_j$, or else to determine them by some sort of clustering algorithm. In a SLM, both the centers $a_i$ as well as the widths $b_i$ of the Gaussians are taken from a continuous probability distribution, and are therefore independent of the training samples
\footnote{The probability distribution of the parameters $a_i$ and $b_i$ can be any continuous probability distribution, for example the normal or the uniform distribution. However, the distribution of $b_i$ should be such that $P(b_i \leq 0) = 0$, since $b_i \in \mathbb R_{>0}$. In theory this is the only constraint, but in practice there might be more constraints due to rounding errors in the numerical evaluations of the Gaussians.}%
. This has the advantage that no prior information about the training samples is necessary, and that it takes practically no time to determine the parameters.

\subsection{TSM interpretation}\label{sec:TSM}

The generalised RBF network has been shown to be equivalent to the Takagi-Sugeno fuzzy model~\cite{takagi1985fuzzy} under some circumstances~\cite{hunt1996extending}. The Takagi-Sugeno model (TSM) consists of a number of fuzzy if-then rules with a fuzzy premise and a linear consequence, e.g.
\begin{eqnarray}
	R_i & : & \mathrm{if}\  x_1 \ \mathrm{is} \ A_1 \ \wedge \ x_2 \ \mathrm{is} \ A_2 \ \wedge \ \cdots \ \wedge \ x_m \ \mathrm{is} \ A_m \nonumber\\
	& & \mathrm{then}  \ \theta_i(x) = \alpha_i x + \beta_i
\end{eqnarray}
	
The circumstances under which the TSM is equivalent to the generalised RBF network, and therefore also to the SLM, are: $1)$ the number of fuzzy rules is equal to the number of RBF units, or local models for the SLM, $2)$ the membership functions within each rule are all Gaussians, and $3)$ the operator used for the fuzzy $\wedge$ is multiplication. Using multiplication for the if-then part as well, and letting $g_{ik}(x)$ denote the fuzzy membership function of the fuzzy set $A_k$ of fuzzy rule $R_i$, this result is straightforward:
\begin{eqnarray}
	TSM(x) & = & \sum_{i=1}^{\order} \prod_{k=1}^n{g_{ik}(x_k)} (\alpha_i x + \beta_i)\nonumber\\
	& = & \sum_{i=1}^{\order} \prod_{k=1}^n{e^{-b_{i}(x_k-a_{ik})^2}} (\alpha_i x + \beta_i)\nonumber\\
	& = & \sum_{i=1}^{\order} e^{-b_{i}\sum_{k=1}^n{(x_k-a_{ik})^2}} (\alpha_i x + \beta_i)\\
	& = & \sum_{i=1}^{\order} e^{-b_{i}||x-a_i||^2} (\alpha_i x + \beta_i)\nonumber\\
	& = & SLM(x).\nonumber
\end{eqnarray}
A SLM can therefore be seen as a TSM with randomly generated Gaussian membership functions for the premises of the fuzzy rules. This equivalence relation gives the following result:

\begin{cor}
	Theorem \ref{thm:main} also holds for the Takagi-Sugeno fuzzy model with $\order$ fuzzy rules, if the membership functions for each rule are Gaussians with parameters chosen randomly from continuous probability distributions, and the operator used for the fuzzy $\wedge$ is multiplication.
\end{cor}

\subsection{LLM interpretation}\label{sec:LLMint}

LLMs like in definition \ref{def:LLM} usually appear in the context of nonlinear dynamical systems, where a complex (nonlinear) system is approximated by taking a weighted sum of simpler (linear) models. These linear models can be seen as linearisations of the nonlinear system around an operating point, and the weights are used to describe the transitions between the operating points. The problem of using a LLM to approximate a nonlinear dynamical system, belongs to the area of system identification. Several overviews of identification techniques for local models exist~\cite{murray1997multiple,verdult2002non}.

The input of such a \emph{dynamical} LLM consists of past inputs and outputs or states of the system, while the output of the dynamical LLM should give the next output or state of the system, e.g.
\begin{equation}
	s_{t+1} = f(s_t,u_t),
\end{equation}
where $s$ is the output or state, and $u$ the control input of the dynamical system. If the function $f$ is differentiable, it can be linearised around an operating point $(s^{(i)},u^{(i)})$, to get: 
\begin{eqnarray}
	f(s_t,u_t) & \approx & f(s^{(i)},u^{(i)}) + \frac{\partial f}{\partial s_t}(s^{(i)},u^{(i)}) (s_t - s^{(i)})\nonumber\\
						& & + \frac{\partial f}{\partial u_t}(s^{(i)},u^{(i)}) (u_t - u^{(i)})\nonumber \\
	& = & A_i s_t + B_i u_t + O_i,
\end{eqnarray}
where $A_i$ and $B_i$ are the partial derivatives of $f$ w.r.t. $s$ and $u$ respectively, and $O_i = f(s^{(i)},u^{(i)}) - A_i s^{(i)} - B_i u^{(i)}$. The local linear model is then a weighted sum of these linearisations:
\begin{equation}\label{eq:dLLM}
	LLM(s_t,u_t) = \sum_{i=1}^{\order}{g_i(s_t,u_t) (A_i s_t + B_i u_t + O_i)},
\end{equation}
where $g_i$ is a scalar-valued function that represents the validity of the local model. The relationship between these kinds of dynamical LLMs and TSMs has also been investigated~\cite{johansen2000interpretation}.

Going out of the area of dynamical systems, if the input of the function $f$ is denoted as $x$ rather than $(s_t,u_t)$, the LLM of equation \ref{eq:dLLM} is equal to the one from Definition \ref{def:LLM}.

Now, for a SLM, the weights are Gaussian functions of the input, but the parameters of the Gaussians (i.e. their centers and widths, or the operating points and regimes) are chosen randomly. This implies that the local models are supposed to be good approximations of the global system around randomly chosen operating points, while the traditional approach is to use a gradient-descent type algorithm to determine the operating points and regimes. Theorem \ref{thm:main} shows that even with randomly chosen operating points and regimes, it is possible to approximate any function with any desired accuracy, as long as the number of local models $\order$ is high enough, and that this can be done with a direct algorithm.


\section{Experiments and results}
\label{sec:exp}

As an example, in this section the SLM algorithm is compared with the ELM algorithm for modelling the following dynamical system:

\begin{eqnarray}
	\frac{d x_1}{d t} & = & x_2,\nonumber\\
	\frac{d x_2}{d t} & = & \lambda(x_1^2 - 1)x_2 - x_1.
\end{eqnarray}

These are the equations of a Van der Pol oscillator. It is a nonlinear system with a limit cycle, a phenomenon that does not occur in linear systems. The experiment is implemented in Matlab, on a computer with a $2.80$ GHz processor.

\subsection{Data}
The model inputs are the $x_1$ and $x_2$ values of the Van der Pol oscillator with $\lambda=1$ for timesteps $t=1,\ldots, 1000$, while the desired model outputs are the $x_1$ and $x_2$ values for timesteps $t+1$. The values are computed with the Euler forward method. The data is seperated in three phases: the learning phase, the generalisation phase, and the simulation phase. In each phase, the dynamical system is run for $10$ different initial conditions, giving $10^4$ data samples.

\subsection{Models used}

The data is tested on a SLM with $100$ local models. The parameters $a_i$ are drawn from a zero-mean normal distribution with variance $2$, while the parameters $b_i$ are drawn from a uniform distribution over $(0,1)$. The ELM uses the same distributions for the parameters, but with $300$ hidden neurons since $\order (n+1) = 300$ in this case (see Section \ref{sec:elmint} for a comparison between the number of hidden neurons in an ELM and the number of local models in a SLM).

The experiment consists of three phases: a learning phase, a generalisation phase and a simulation phase. In the learning phase, the SLM is trained using the algorithm from Figure \ref{fig:alg}, and the ELM is trained using the ELM algorithm proposed by Huang et al.~\cite{huang2006extreme}. In the generalisation phase, only the inputs of the system are used, and the models have to produce the outputs using the model parameters obtained during the training phase. In the simulation phase, only the initial conditions of the system are given, and the models have to use the model output of the previous timestep as an input for future timesteps, also using the model parameters obtained during the training phase.

\subsection{Results and discussion}

The whole process has been repeated $100$ times. The average results are shown in Table \ref{tab:results}, where MSE stands for Mean Squared Error. The output of the training phase of the last run is shown in Figure \ref{fig:lastrunof50RLM}, for the SLM case only. The ELM case showed similar results, as well as the generalisation and simulation phases for both algorithms.

From the results we see that the SLM algorithm takes less computation time. As expected, this is not due to the calculation of the pseudo-inverse, which has actually become slower in the SLM algorithm (the reason for this is unclear). The construction of the matrix that is to be inverted has been done more efficiently in the SLM approach by making use of its block structure.

The errors during all three phases are similar in both methods, while the number of local models for the SLM was three times less than the number of hidden neurons for the ELM. This is in line with the theory presented in this paper.\\

\begin{table}[htb]
	\caption{Comparing the errors and computation time of the SLM and ELM algorithm for the modelling of a Van der Pol oscillator.}
	\label{tab:results}
	\centering
		\begin{tabular}{|p{0.3\textwidth}|p{0.17\textwidth}|p{0.17\textwidth}|}
			\hline
			& Average over 100 SLM tests & Average over 100 ELM tests\\
			\hline
			Total computation time & $9.0952$ sec. & $17.8737$ sec.\\
			\hline
			Computation time for pseudo-inverse & $1.8038$ sec.  & $1.5514$ sec.\\
			\hline
			MSE during training & $2.3872 \cdot 10^{-11}$ & $1.3108 \cdot 10^{-11}$\\
			\hline
			MSE during generalisation & $2.2645 \cdot 10^{-10}$ & $3.1529 \cdot 10^{-9}$\\
			\hline
			MSE during simulation & $7.6113 \cdot 10^{-5}$ & $7.5266 \cdot 10^{-5}$\\
			\hline
		\end{tabular}
\end{table}

\begin{figure}[htb]
	\centering
		\includegraphics[width=0.99\textwidth]{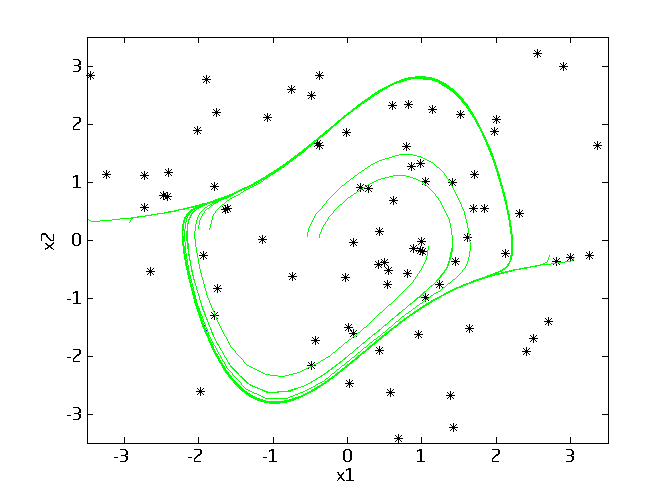}
	\caption{Modelling a Van der Pol oscillator with a Random Local Linear Model. The black dots are the locations of the centers of the Gaussians (parameters $a_i$), which are distributed randomly. Around each black dot, a certain linear model is valid. Adding the linear models together in a nonlinear fashion results in a global model that accurately approximates the nonlinear dynamics of the oscillator.} 
	\label{fig:lastrunof50RLM}
\end{figure}

\section{Conclusions}
\label{sec:concl}

An extension to Extreme Learning Machines (ELM) was proposed in this paper, together with an algorithm and theorems concerning the approximation capabilities of the Random Local Linear Model (SLM). Under certain conditions, the SLM can achieve similar approximation capabilities as the ELM while using less local models than the number of hidden nodes of an ELM. Although this is not necessarily an advantage, the SLM allows for a clear physical interpretation and can be run more efficiently than an ELM when the number of local models and hidden nodes is high.

The SLM is similar to Local Linear Models from the area of system identification and, under some circumstances, to the well-known Takagi-Sugeno Model from the area of fuzzy modelling, for which an upper bound on the number of randomly generated rules is given. This also makes them more fit to be used for modelling dynamical systems. Compared to these techniques, the training of a SLM can be done more efficiently, since the operating points or fuzzy rules can be chosen randomly and independent from the input, which implies that less prior knowledge is required. The training algorithm then uses linear regression, which can be performed much faster than most other parameter estimation methods.

The theorems and proofs of this paper could be extended to more general cases. However, the link with system identification and fuzzy modelling techniques already makes sure that the SLM could be an alternative to some very popular techniques, and it allows much room for further investigation.

\section*{Acknowledgment}

This research was performed at Almende B.V. in Rotterdam, the Netherlands. The author would like to thank Anne van Rossum and Dimitri Jeltsema for their supervision during this research, and Giovanni Pazienza for providing valuable feedback on this paper.





\bibliographystyle{elsarticle-num}
\bibliography{mybiblio}







\bigskip

Laurens Bliek received the B.Sc. degree in applied mathematics at Delft University of Technology in 2011, and is currently graduating for the M.Sc. degree of the same study. His minor, B.Sc thesis, and many of his electives, had to do with artificial intelligence and neural networks. His research interests include neural networks, dynamical systems theory and optimisation.

\begin{figure}[htbp]
	\centering
		\includegraphics[width = 0.2\textwidth]{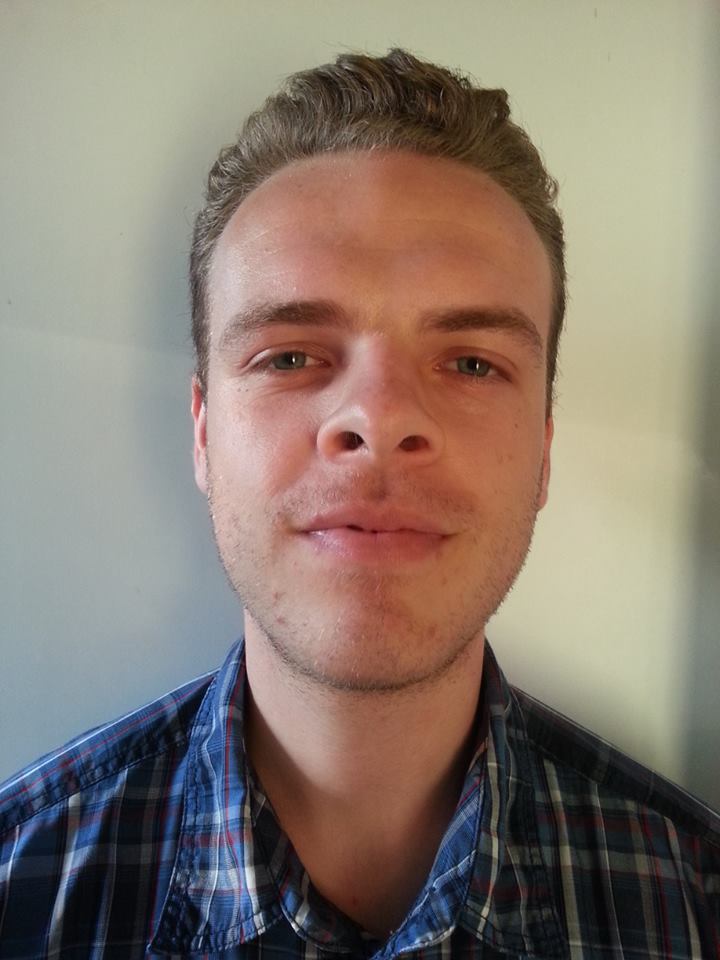}
	\label{fig:Laurens2013}
\end{figure}

\end{document}